\documentclass[11pt]{amsart}
\usepackage{FrankTall}
\usepackage[utf8]{inputenc} 
\usepackage{mathtools,amssymb,amsxtra,mathrsfs}
\usepackage{bm} 
\usepackage{xcolor,enumitem}
\usepackage{ragged2e}

\usepackage{indentfirst}

\newcommand{\w}{\omega}

\renewcommand{\op}{\operatorname}

\renewcommand{\set}[1]{\{#1\}}
\renewcommand{\abs}[1]{\left| #1 \right|}
\renewcommand{\st}{\,\mathrm{:}\,}

\newcommand{\R}{\mathbb{R}}

\let\phi\varphi

\renewcommand{\bar}{\overline}

\usepackage{graphicx} 
\newcommand\rsetminus{\mathbin{\mathpalette\rsetminusaux\relax}}
\newcommand\rsetminusaux[2]{\mspace{-3mu}\raisebox{\rsmraise{#1}\depth}{\rotatebox[origin=c]{-35}{$#1\smallsetminus$}}\mspace{-3mu}}
\newcommand\rsmraise[1]{%
  \ifx#1\displaystyle 1 \else
    \ifx#1\textstyle .3 \else
      \ifx#1\scriptstyle .2 \else
        .1%
      \fi
    \fi
  \fi}

\let\setminus\bigsetminus

\theoremstyle{plain}      

\theoremstyle{definition} 
\newtheorem{thrm}{Theorem}[section]
\newtheorem*{thrm*}{Theorem}

\newtheorem{lemma}[thrm]{Lemma}
\newtheorem*{lemma*}{Lemma}
\newtheorem{propn}[thrm]{Proposition}
\newtheorem*{propn*}{Proposition}
\newtheorem{defn}{Definition}[section]
\newtheorem*{defn*}{Definition}

\theoremstyle{remark}     

\newtheoremstyle{cited}
  {3pt}
  {3pt}
  {}
  {}
  {\bfseries}
  {.}
  {.5em}
  {\thmname{#1} \thmnumber{#2} \thmnote{\normalfont#3}}

\theoremstyle{cited}

\begin{document} 
\title[Pseudocompactness and the UMP]{Pseudocompactness and the Uniform Metastability Principle in Model Theory}
\author[Hamel]{Clovis Hamel$^1$}
\author[Tall]{Franklin D. Tall$^1$}
\thanks{$^1$Research supported by NSERC grant A-7354.}

\date{\today}



\begin{abstract} We prove that uniform metastability is equivalent to all closed subspaces being pseudocompact and use this to provide a topological proof of the metatheorem introduced by Caicedo, Due\~{n}ez and Iovino on uniform metastability and countable compactness for logics.  

\end{abstract}
\maketitle

\section{Introduction}
Terence Tao \cite{Tao2008} introduced the notion of \textit{metastability}: for a single sequence, being metastable is just to be Cauchy. However, \textit{uniform metastability} is an intermediate form of convergence for families of sequences, in between uniform convergence and pointwise convergence. Tao has used metastability to solve problems in ergodic theory \cite{Tao}. Eduardo Due\~{n}ez and Jos\'{e} Iovino realized the logical nature of metastability and have been applying it extensively \cite{Duenez2017}. Due\~{n}ez, Iovino and Xavier Caicedo introduced the \textit{Topological Uniform Metastability Principle} \cite{Caicedo2019} and proved that it holds for countably compact spaces. Then they proved that the converse implication holds for the space of structures of continuous logic. The proof is long and model-theoretic, so we were inspired to find a purely topological short proof, which we present here, and from which the model theoretic result follows at once. This is likely the first application of \textit{pseudocompactness} to model theory. Since our intended audience is model theorists, we include proofs of relevant material on pseudocompactness scattered through the topological literature.  

\section{A Brief Introduction to Metastability}

\begin{defn}
A \textit{sampling of $\omega$} is a family $\set{\eta_n : n<\omega}\subseteq [\omega]^{<\omega}$ such that $\eta_n\subseteq \omega\setminus n$ for each $n<\omega$. Let $\mathcal{S}$ denote the set of all samplings of $\omega$.  
\\ Let $(X,d)$ be a metric space. A sequence $\langle x_n : n<\w \rangle$ is \textit{metastable} if for each $\varepsilon>0$ and each sampling $\eta$, there is $m<\omega$ such that $(\forall i,j\in \eta_m)\ (d(x_i,x_j)<\varepsilon)$. 
\end{defn}


It was proved by Due\~{n}ez and Iovino \cite{Duenez2017} that a sequence is metastable if and only if it is Cauchy. The relevant distinction occurs when one considers uniform metastability:

\begin{defn}
A family $A\subseteq X^\omega$, where $(X,d)$ is a metric space, is \textit{uniformly metastable} if there is a family $\set{E_{\varepsilon,\eta} :   \varepsilon>0 ,\ \eta \in \mathcal{S}}$ such that whenever $\eta\in \mathcal{S}$ and $\varepsilon>0$, each sequence in $A$ is metastable witnessed by the same $m<E_{\varepsilon,\eta}$. 
A sequence of functions $\langle f_n : n<\omega \rangle$ in $R^X$ is \textit{uniformly metastable} if there is a family $\set{E_{\varepsilon,\eta} :   \varepsilon >0 ,\ \eta \in \mathcal{S}}$ such that whenever $\eta\in \mathcal{S}$ and $\varepsilon >0$, for each $x\in X$ the sequence $\langle f_n(x) : n<\omega \rangle$ is metastable witnessed by the same $m<E_{\varepsilon,\eta}$.
\end{defn}

The following examples from an early version of \cite{Caicedo2019} show that uniform metastability is strictly in between uniform convergence and pointwise convergence:
\begin{itemize}
    \item The family of all eventually $0$ sequences in $2^\omega$ is not uniformly metastable even though each sequence is trivially convergent. To see this, take the subfamily of all sequences with arbitrarily long initial segments with alternated $0$'s and $1$'s and $\eta_n=\set{n,n+1}$.
    \item The set of all monotonic sequences in $2^\w$ is uniformly metastable witnessed by $E_{\varepsilon,\eta}=\max{\eta_0}$. However, the convergence is not uniform. 
\end{itemize}

It is a natural to ask when results regarding pointwise convergence of functions can be improved to uniform metastability in a way similar to that of Tao's metastable dominated convergence theorem \cite{Tao2008}. In \cite{Caicedo2019}, a topological proof is given for the following fact: if $X$ is countably compact, then on any closed subspace, there is no distinction between pointwise convergence and uniform metastability. The converse result is only proved in \cite{Caicedo2019} in a model theoretic setting using powerful machinery. We produced a topological proof of this converse result using the following fact: a countably compact space is a space with every closed subspace pseudocompact. The model theoretic result follows at once from this topological fact and a few basic remarks. 

\begin{defn}
A topological space $X$ is \textit{pseudocompact} if every continuous real-valued function on $X$ has bounded image. 
\end{defn}

There is a whole book devoted to pseudocompact spaces \cite{Hrusak2018}. The following basic result can also be found in \cite{Tkachuk2011}:

\begin{propn}
A completely regular space $X$ is pseudocompact if and only if every locally finite family of non-empty open sets (i.e. every point of $X$ has a neighbourhood meeting at most finitely many members of the family) is finite.
\end{propn}

\begin{proof}
Suppose $X$ is pseudocompact and that there is an infinite locally finite family of non-empty open sets $\set{U_n : n<\omega}$. Take $x_n\in U_n$ for each $n<\omega$. By complete regularity, take a continuous $f_n:X\to\R$ such that $f_n(x_n)=n$ and $f\restriction X\setminus U_n=0$. Then $F=\sum_{n<\omega}f_n$ is continuous since $\set{U_n : n<\omega}$ is locally finite: given $x\in X$, let $S_x=\set{n<\omega : x\in U_n}\in [\omega]^{\omega}$; the continuity of $F$ at $x$ follows from $\bigcup_{n\in S_x}{\bar U_n}=\Bar{\bigcup_{n\in S_x}U_n}$. 
Conversely, suppose $X$ is not pseudocompact, then there is an unbounded continuous function $f:X\to \R$. Since $f^2$ is also continuous and unbounded, we can assume $f\geq 0$. Let $x_0\in f[X]$; if $x_n\in f[X]$ has been constructed, take $x_{n+1}\in f[X]$ such that $f(x_{n+1})>f(x_n)+1$. If we denote the ball of center $f(x_n)$ and radius $1$ by $B(f(x_n),1)$, then $\mathcal{B}=\set{f^{-1}[B(f(x_n),1)] : n<\omega}$ is an infinite family of pairwise disjoint non-empty open sets. Now suppose $\mathcal{B}$ is not locally finite, then there is a point $x\in X$ such that every open neighbourhood of $x$ contains elements with arbitrarily large images, contradicting the continuity of $f$.
 
\end{proof}

\begin{rmk*}
Notice that when pseudocompactness fails, one can get the infinite locally finite family of open sets to be pairwise disjoint. Also notice that complete regularity is unnecessary for the direction ``every locally finite family of open sets is finite'' implies pseudocompactness. However, regularity is required. 
\end{rmk*}


The following proposition follows from a theorem and an exercise in \cite{Engelking1989};

\begin{propn}
A space is countably compact if and only if every closed subspace is pseudocompact.
\end{propn}
\begin{proof}
If $X$ is countably compact and there is a closed subspace $C\subseteq X$ that is not pseudocompact then, as in the proof of Proposition 2.1, $C$ includes a closed discrete subspace and so does $X$, which contradicts countable compactness. Conversely, if $X$ is not countably compact, it includes a discrete closed set $C=\set{x_n : n<\omega}$. Letting $f(x_n)=n$, we obtain a continuous unbounded function on $C$.
\end{proof}

Now we present the connection between pseudocompactness and uniform metastability:

\begin{propn}
Let $X$ be a regular topological space. If every sequence of continuous real-valued functions $\langle f_n : n<\omega \rangle$ on $X$ that converges pointwise is uniformly metastable, then $X$ is pseudocompact.
\end{propn}

\begin{proof}
Suppose $X$ is not pseudocompact and let $\mathcal B =\set{U_n : n<\omega}$ be a infinite locally finite family of pairwise disjoint non-empty open sets. For each $n<\omega$, take $x_n\in U_n$. Then consider the functions $f_n:X\to \R$ such that $f(x_n)=1$ and $f\restriction X\setminus U_n=0$. Then the function $F=\sum_{n<\omega}f_n$ is continuous since $\mathcal B$ is locally finite. Consider $g_n=\sum_{i\leq n}f_i$. Then the sequence of continuous functions $\langle g_0, F, g_1, F, g_2, F,\ldots \rangle$ converges pointwise to $F$ but it is not uniformly metastable as it contains all eventually $1$ sequences with arbitrarily long initial segments of alternating $0$'s and $1$'s.    
\end{proof}

\begin{defn}
The \textit{Topological Uniform Metastability Principle} holds for a topological space $X$ if whenever a sequence of real-valued continuous functions converges pointwise on a closed subspace $C\subseteq X$, it is uniformly metastable on $C$.
\end{defn}

The previous results allow us to characterize the equivalence between the topological uniform metastability principle and pointwise convergence.   

\begin{thrm} 
Let $X$ be a completely regular space. Then $X$ is countably compact if and only if the topological uniform metastability principle holds for $X$. 
\end{thrm}
\begin{proof}
We reproduce the proof given in an early version of \cite{Caicedo2019} when $X$ is countably compact: assume uniform metastability fails and let $\varepsilon>0$ and $\eta \in \mathcal{S}$ be witnesses of this fact. Then, for each $n<\omega$, there is $x\in X$ such that for each $k<n$, $M_{x,k}=\max{\set{\abs {f_i(x)-f_j(x)}: i,j\in \eta_k}}\geq\varepsilon$. Then $x\in \Insect_{k\leq n} A_k$ where $A_k=\set{z\in X : M_{z,k}\geq\varepsilon}$ is closed by the continuity of the $f_n$'s. Thus $\set{A_k : k<\omega}$ is centred and by countable compactness, there is $x\in \Insect_{k<\omega} A_k$, which contradicts the convergence of $\langle f_n(x) : n<\omega \rangle$.
Conversely, suppose $X$ is not countably compact. Then, by Proposition 2.2, there is a closed subspace $C\subseteq X$ that is not pseudocompact and so, by Proposition 2.3, there is a sequence of continuous real-valued functions on $X$ that converges pointwise on $C$ but is not uniformly metastable on $C$.
\end{proof}

\begin{rmk*}
The current version of \cite{Caicedo2019} proves the previous equivalence assuming that $X$ is regular and paracompact. We just showed that the paracompactness assumption can be replaced by  assuming that $X$ is completely regular. Also, \cite{Caicedo2019} points out that the analogue of the uniform metastability principle for nets, instead of sequences, is equivalent to $X$ being compact. 
\end{rmk*}

\section{The Uniform Metastability Principle}

Logics for metric structures are properly presented in \cite{Eagle2017} and \cite{Caicedo2017}. Given a logic for metric structures $\mathcal{L}$ and a language $L$, recall that the topology on the space of $L$-structures $Str(L)$ is determined by the basic closed sets $[\varphi]=\set{\mathfrak{M}\in Str(L) : \mathfrak{M}\models \varphi}$. We regard $L$-sentences as continuous $[0,1]$-valued functions on the space of $L$-structures in the natural way: $\mathfrak{M} \mapsto \varphi^\mathfrak{M}$. In this context, we now define the model theoretic analogue of metastability:

\begin{defn}
Let $\mathcal{L}$ be a logic for metric structures and $L$ a language. Given an $L$-theory $T$, we say that a sequence of $L$-sentences $\langle \varphi_n : n<\omega \rangle$ \textit{converges pointwise modulo $T$} if and only for for every model $\mathfrak{M}$ of $T$, the sequence $\langle {\varphi_n}^\mathfrak{M} : n<\omega \rangle$ converges. We say that the sequence is \textit{uniformly metastable modulo $T$} if the family $\set{\langle {\varphi_n}^\mathfrak{M} : n<\omega \rangle : \mathfrak{M}\models \varphi}$ is uniformly metastable.
\end{defn}

\begin{defn}
The Uniform Metastability Principle (UMP) for a logic $\mathcal{L}$ is the following statement: "if $L$ is a vocabulary and $T$ is an $L$-theory, then every sequence of $L$-sentences $\langle \varphi_n : n<\omega \rangle$ that converges pointwise modulo $T$ is also uniformly metastable modulo $T$".
\end{defn}

In an early version of \cite{Caicedo2019}, it was proved that the UMP is equivalent to the logic being countably compact. This follows from the following two lemmas:

\begin{lemma}
The logic topology is completely regular.
\end{lemma}
\begin{proof}
Let $C\subseteq Str(L)$ and $\mathfrak M\notin C$. Then there must be a formula $\varphi$ such that $\varphi^{\mathfrak M}<1$ and $(\forall \mathfrak N\in C)\ \varphi^{\mathfrak N}=1$ (as otherwise $\mathfrak M$ would belong to $C$ by the definition of the topology on $Str(L)$). Then $\varphi$ is the continuous function that separates $C$ and $\mathfrak{M}$.
\end{proof}

\begin{lemma}
The closed subspaces of the logic topology are completely determined by $L$-theories, i.e. $C\subseteq Str(L)$ is closed if and only if there is an $L$-theory $T$ such that $C$ is the set of $L$-structures that are models of $T$. 
\end{lemma}
\begin{proof}
Suppose $C$ is a closed set in $Str(L)$, then $C$ is the intersection of basic closed sets, say $C=\Insect_{\alpha<\kappa}[\varphi_\alpha]$. Thus $C$ is the set of $L$-structures that are models of the theory $T=\set{\varphi_\alpha : \alpha<\kappa}$. Conversely, each model of an $L$-theory $T$ belongs to the intersection of all $[\varphi]$ where $\varphi$ ranges over $T$. 
\end{proof}

\begin{defn}
A \textit{logic $\mathcal{L}$ is countably compact} if and only if given a language $L$, the space of $L$-structures $Str(L)$ is countably compact. 
\end{defn}

Putting all this together, we easily obtain the main result of the early version of \cite{Caicedo2019}:

\begin{thrm}
Let $\mathcal{L}$ be a logic for metric structures. The UMP holds if and only if $\mathcal{L}$ is countably compact. 
\end{thrm}

\newpage
\bibliographystyle{alpha}
\bibliography{refdb2020.bib}

\providecommand{\noopsort}[1]{}
\begin{thebibliography}{HTMT18}

\bibitem[Cai17]{Caicedo2017}
X.~Caicedo.
\newblock Maximality of continuous logic.
\newblock In J.~Iovino, editor, {\em Beyond first order model theory}, pages
  105--130. CRC Press, Boca Raton, FL, 2017.

\bibitem[CDI19]{Caicedo2019}
X.~Caicedo, E.~Due{\~n}ez, and J.~Iovino.
\newblock Metastable convergence and logical compactness.
\newblock arXiv:1907.02398, 2019.

\bibitem[DI17]{Duenez2017}
E.~Due{\~n}ez and J.~Iovino.
\newblock Model theory and metric convergence {I}: {M}etastability and
  dominated convergence.
\newblock In J.~Iovino, editor, {\em Beyond first order model theory}, pages
  131--187. CRC Press, Boca Raton, FL, 2017.

\bibitem[Eag17]{Eagle2017}
C.~J. Eagle.
\newblock Expressive power of infinitary {$[0, 1]$}-logics.
\newblock In J.~Iovino, editor, {\em Beyond first order model theory}, pages
  3--22. CRC Press, Boca Raton, FL, 2017.

\bibitem[Eng89]{Engelking1989}
R.~Engelking.
\newblock {\em General {T}opology}.
\newblock Heldermann Verlag, Berlin, 1989.

\bibitem[HTMT18]{Hrusak2018}
M.~Hru{\v{s}}{\'a}k, {\'A}.~Tamariz-Mascar{\'u}a, and M.~Tkachenko, editors.
\newblock {\em Pseudocompact Topological Spaces: A Survey of Classic and New
  Results with Open Problems}, volume~55 of {\em Developments in Mathematics}.
\newblock Springer, 2018.

\bibitem[Tao]{Tao}
T.~Tao.
\newblock Walsh's ergodic theorem, metastability, and external {C}auchy
  convergence.
\newblock http://terrytao.wordpress.com.

\bibitem[Tao08]{Tao2008}
T.~Tao.
\newblock Norm convergence of multiple ergodic averages for commuting
  transformations.
\newblock {\em Ergodic Theory and Dynamical Systems}, 28(2):657--688, 2008.

\bibitem[Tka15]{Tkachuk2011}
V.~Tkachuk.
\newblock {\em A $C_p$-theory problem book}.
\newblock Problem Books in Mathematics. Vol. I-IV. Springer, 2011-2015.

\end{thebibliography}

\end{document}